\documentclass[11pt,twoside,a4paper]{article}

\usepackage{amsthm,amsmath,amssymb,amsfonts,latexsym,enumerate,vmargin,hyperref,graphicx}
\setmarginsrb{20mm}{20mm}{20mm}{20mm}{0pt}{10mm}{0pt}{10mm}

\newtheorem{thm}{Theorem}
\newtheorem*{thm*}{Theorem}
\newtheorem{lemma}[thm]{Lemma}
\newtheorem*{lemma*}{Lemma}
\newtheorem{prop}[thm]{Proposition}
\newtheorem*{prop*}{Proposition}
\newtheorem{cor}[thm]{Corollary}
\newtheorem*{cor*}{Corollary}

\newtheorem*{claim*}{Claim}

\newtheorem*{observ*}{Observation}

\newtheorem*{fact*}{Fact}
\newtheorem{conj}[thm]{Conjecture}
\newtheorem*{conj*}{Conjecture}
\newtheorem{ex}[thm]{Example}
\newtheorem*{ex*}{Example}

\theoremstyle{remark}

\newtheorem*{rmk*}{Remark}

\newtheorem*{rmks*}{Remarks}

\newtheorem*{nota*}{Notation}

\newtheorem*{dfn*}{Definition}

\newtheorem*{dfns*}{Definitions}

\def\Z{\mathbb{Z}}
\def\N{\mathbb{N}}

\def\R{\mathbb{R}}

\newcommand\ip[1]{\left<#1\right>}

\begin{document}

\title{On Sums of Generating Sets in $\Z_2^n$}
\author{Chaim Even-Zohar 
\thanks{
Einstein Institute of Mathematics, The Hebrew University, Jerusalem 91904, Israel, 
chaim.evenzohar@mail.huji.ac.il.
This paper is based on the author's MSc thesis, under the supervision of Prof. Nati Linial.
}}
\maketitle

\abstract{
Let $A$ and $B$ be two affinely generating sets of $\Z_2^n$. As usual, we denote their Minkowski sum by $A+B$.
How small can $A+B$ be, given the cardinalities of $A$ and $B$? We give a tight answer to this question.
Our bound is attained when both $A$ and $B$ are unions of cosets of a certain subgroup of $\Z_2^n$.
These cosets are arranged as Hamming balls, the smaller of which has radius $1$.

By similar methods, we re-prove the Freiman--Ruzsa theorem in $\Z_2^n$, with an optimal upper bound.
Denote by $F(K)$ the maximal spanning constant $|\langle A \rangle| / |A|$ 
over all subsets $A \subseteq \Z_2^n$ with doubling constant $|A+A|/|A| \leq K$. 
We explicitly calculate $F(K)$, and in particular show that $4^K / 4K \leq F(K)\cdot(1+o(1))  \leq 4^K / 2K $.
This improves the estimate $F(K) = poly(K) 4^K$, found recently by Green and Tao \cite{green_tao} and by Konyagin \cite{konyagin}.
}

\section{Introduction}

Much work has been devoted to the study of Minkowski sums of sets.
Questions concerning such sums come up in geometry, and are at the core of additive combinatorics. 
Research in this area has blossomed in recent years, 
and even Tao and Vu's monograph~\cite{tao_vu} no longer covers all the most recent developments. 
In this paper we concentrate on the Minkowski sum of two generating sets of~$\Z_2^n$.

We first review some of the relevant literature. 
Let $G$ be an abelian group, and let $A$ and $B$ be two finite subsets of $G$. 
As usual, we denote
$$ A+B = \{a+b \;|\; a \in A, b \in B \}$$
and we ask about the minimum of $|A+B|$, given the cardinalities of $A$ and $B$.

In general, the answer ranges from $\max(|A|,|B|)$ to $|A|+|B|-1$, depending on the structure of $G$. 
For a torsion-free $G$, if $A$ and $B$ are arithmetic progressions with the same step,
then $|A+B| = |A|+|B|-1$, which is optimal. 
Likewise, if $G = \Z_p$ is cyclic of prime order, 
then the answer is given by the Cauchy--Davenport theorem, 
$|A+B| \geq \min(|A|+|B|-1,|G|)$~\cite{cauchy,davenport}.
Moreover, by a theorem of Vosper~\cite{vosper}, if $|A|+|B|<|G|$ then equality holds only for arithmetic progressions.
In the other extreme case, $G$ has a finite subgroup of a suitable cardinality. 
Thus, if $H \triangleleft G$ is a subgroup of cardinality $|H| = \max\left(|A|,|B|\right)$, 
an optimal choice is to have $A$ and $B$ be subsets of $H$, in which case $|A+B| = \max\left(|A|,|B|\right)$.
More generally, $|A+B|$ can be as small as $\max\left(|A|,|B|\right)$ if and only if 
$\min(|A|,|B|) \leq |H|$ and $|H|$ divides $\max(|A|,|B|)$~\cite[p. 55]{tao_vu}.
In the general case~\cite{eliahou_kervaire,eliahou_kervaire_plagne}, 
the smallest possible cardinality of $|A+B|$ is
$\min \left(\left\lceil |A|/|H|\right\rceil + \left\lceil |B|/|H|\right\rceil - 1 \right)|H|$,
where the minimum is over all finite subgroups $H$ of $G$.
In a sense, this result interpolates between the two extremes. 
In an optimal construction~\cite{bollobas_leader,eliahou_kervaire_plagne} the sets $A$ and $B$ 
are contained in $\left\lceil |A|/|H|\right\rceil$ and $\left\lceil |B|/|H|\right\rceil$ cosets of $H$,
whose arrangement is a lexicographical variant of an arithmetic progression.
In particular, for $G$ a $2$-torsion group this reduces to the well-studied Hopf--Stiefel function
~\cite{hopf,stiefel,yuzvinsky,bollobas_leader,eliahou_kervaire_vector_spaces,eliahou_kervaire_hs}.

{\em Stability} is a recurring theme in modern extremal combinatorics. 
Once an extremal problem is solved, it is interesting to explore 
what happens when we consider candidate solutions that do not resemble the global optimum. 
The crucial feature of the above-mentioned optimal constructions is that
$A$ and $B$ are densely packed in cosets of properly chosen subgroups of $G$. 
We therefore return to the original question,
under the requirement that $A$ and $B$ are not allowed to be contained in a proper subgroup of $G$ or a coset thereof.
The \emph{affine span} of $A$, denoted $\left\langle A \right\rangle$,
is the smallest coset (of any subgroup) containing $A$.
We say that $A$ \emph{affinely generates} $G$ if $\left\langle A \right\rangle = G$.
Clearly this definition coincides with the usual notion of a generating set if $0 \in A$.
The refined problem is as follows:
In a finitely generated abelian group $G$, find $\min|A+B|$ as a function of $|A|$ and $|B|$,
where $A$ and $B$ are finite affinely generating subsets of $G$.

Naturally, the structural properties of $G$ play a role in this problem as well.
For the torsion-free case, $G = \Z^d$, this question and similar ones were discussed by Ruzsa~\cite{ruzsa_zd},
and a full answer was finally given by Gardner and Gronchi~\cite{gardner_gronchi}.
In the extremal construction, the smaller set is a simplex of $d+1$ points,
on one of whose edges lies an arithmetic progression, 
and the other set is roughly the sum of several copies of it. As discussed there,
this is analogous to the Brunn--Minkowski theorem~\cite{schneider}.

Here we present the following lower bound
for the opposite extreme of a $2$-torsion group, $G = \Z_2^n$.

\begin{thm}\label{AB}
Suppose $A,B \subseteq G = \Z_2^n$ such that $\left\langle A \right\rangle = G$, 
$B \neq \varnothing$ and $|A| \leq \frac34 |G|$.\\
If $t$ is the largest positive integer such that
$$ |A| \;\leq\; \frac{t + 1}{2^t} \cdot |G| $$
and $0 \leq k < t$ and $w \in [-1,1]$ are such that
$$ |B| \;=\; \frac{ \binom{t}{0} + \binom{t}{1} + ... + \binom{t}{k} + w \binom{t-1}{k} }{2^t} \cdot |G| $$
then
$$ |A+B| \;\geq\; \frac{\binom{t}{0} + \binom{t}{1} + ... + \binom{t}{k} + \binom{t}{k+1} + w \binom{t-1}{k+1}}{2^t} \cdot |G| $$
This bound is tight when $w=0$, and it is attained by the sets
$$ A = D_1^t \times \Z_2^{n-t} \;\;\;\;\;\;\;\; 
B = D_k^t \times \Z_2^{n-t} \;\;\;\;\;\;\;\; 
A+B = D_{k+1}^t \times \Z_2^{n-t} $$
where $D_k^t = \left\{ x \in \Z_2^t \;\left|\; \#\{i|x_i=1\} \leq k \right. \right\}$ 
is a Hamming ball of radius $k$ in $\Z_2^t$.
\end{thm}

The Freiman--Ruzsa theorem~\cite{ruzsa_thm} is a major result in additive combinatorics.
In the context of the above discussion, it addresses the special case $A=B$.
It states that if $A$ is a subset of an $r$-torsion abelian group with $|A+A| \leq K|A|$,
then $A$ is contained in a coset of cardinality at most $F(K)|A|$, with $F(K) = K^2 r^{K^4}$. 
The special case $r=2$ has received considerable 
attention~\cite{deshouillers_hennecart_plagne,diao,green_ruzsa,green_tao,hennecart_plagne,konyagin,lev_structure,sanders,viola}.
Among the most recent contributions is work by Green and Tao~\cite{green_tao} 
with further improvement by Konyagin~\cite{konyagin}.
It shows that one can take $F(K) = 2^{2K + O(\log K)}$.
Here we exactly determine the lowest possible value of $F(K)$ for $r=2$.

\begin{thm}\label{formulaF}
For $K \geq 1$, denote by $t \geq 1$ the unique integer for which:
$$ \frac{\binom{t}{2} + t + 1}{t + 1} \leq K < \frac{\binom{t+1}{2} + (t+1) + 1}{(t+1) + 1} $$
For $A \subseteq \Z_2^n$ such that $|A+A|/|A| \leq K$, we have $|\langle A \rangle|/|A| \leq F(K)$ where:
$$ F(K) = \begin{cases}
\frac{2^t}{\binom{t}{2} + t + 1} \cdot K \;\;\; & \frac{\binom{t}{2} + t + 1}{t + 1} \leq K < \frac{t^2 + t + 1}{2t} \\
\frac{2^{t+1}}{t^2 + t + 1} \cdot K      \;\;\; & \frac{t^2 + t + 1}{2t} \leq K < \frac{\binom{t+1}{2} + (t+1) + 1}{(t+1) + 1}
\end{cases} $$
This choice of $F(K)$ is tight, and grows as $\Theta(2^{2K}/K)$.
\end{thm}

{\em Compression} is an important tool from extremal set theory.
Much progress in the application of compression to additive problems 
was made by Bollob\'{a}s and Leader in~\cite{bollobas_leader},
and it is a key ingredient in Green and Tao's proof in~\cite{green_tao}.
There is a whole range of compression operators $C$, 
that transform an arbitrary set $A$ to another set $C(A)$, with $|C(A)|=|A|$ and $|C(A)+C(A)| \leq |A+A|$.
By a finite sequence of such compressions, 
it is possible to reduce to the case where $A$ is {\em compressed} in some appropriate sense,
and hence has certain structural properties, which make $A+A$ easier to study.
The difficulty is that $C(A)$ need not be affinely generating even if $A$ is. 
Green and Tao handled this difficulty by restricting the types of compression operators they used. 
Our approach is different. We employ more types of compression operators 
and we proceed as long as possible without jeopardizing affine generation, i.e., as long as
$\left\langle C(A) \right\rangle = \left\langle A \right\rangle$.

{\em Isoperimetric inequalities} play an important role in our work. 
In our investigations of $A+B$,
we prove a new variant of the isoperimetric inequality for the hypercube.

\textbf{Overview:}
In Section~\ref{toolsect} we discuss compressions and other useful tools.
We explore the key notion of compression that maintains affine generation.
In Section~\ref{freisect} Theorem~\ref{formulaF} is proved,
first in an asymptotic form, then with the exact expression.
In Section~\ref{diffsect} we establish Theorem~\ref{AB}.
The proof utilizes our new isoperimetric inequality.

\section{Tools}\label{toolsect}

In this section we briefly survey several concepts and results that are used below.
These include the lexicographic order and the Hopf--Stiefel function.
Then we discuss compressions in $\Z_2^n$, in line with Section $2$ of~\cite{green_tao},
and we introduce the study of compressions that preserve affine generation.

\subsection{The Lexicographic Order}

Throughout, we use the linear basis $\{e_1,e_2,...,e_n\}$ for $\Z_2^n$. 
Elements $x \in \Z_2^n$ are expressed as $x = \sum_{i=1}^n x_i e_i$.
The correspondence between vectors $x \in \Z_2^n$ 
and their supports $\{j\;|\;x_j=1\} \subseteq \{1,...,n\} = [n]$, 
is used to simplify certain notation and arguments.

The \emph{lexicographic order} is a total order on $\Z_2^n$. For $x,y \in \Z_2^n$ we say that $x \prec y$,
if $x_i < y_i$ for the largest coordinate $i$ for which $x_i \neq y_i$.
For example, the ordering of $\Z_2^3$ is: 
$$0 \;\prec\; e_1 \;\prec\; e_2 \;\prec\; e_1+e_2 \;\prec\; 
e_3 \;\prec\; e_1+e_3 \;\prec\; e_2+e_3 \;\prec\; e_1+e_2+e_3$$

The \emph{height}, $\hbar(x)$ of an element $x$ in a finite totally ordered set
is $x$'s place in that order. For a set of elements $A$ we denote 
$\hbar(A) = \sum_{x \in A}\hbar(x)$.

If $T \subseteq \Z_2^n$, then its \emph{initial segment} of size $a$, denoted $IS(a,T)$,
is the set of the $a$ smallest elements of $T$ in the lexicographic order.
We use the abbreviation $IS(a) = IS(a,\Z_2^n)$ for $n \in \N$ large enough.

\subsection{The Hopf--Stiefel Function}

For the reader's convenience we prove the following observation of Bollob\'{a}s and Leader~\cite{bollobas_leader}.

\begin{prop}\label{segs}
For two initial segments $IS(a),IS(b) \subseteq \Z_2^n$, 
the sum $IS(a) + IS(b)$ is an initial segment as well.
\end{prop}

\begin{proof}
For $z \prec x+y$, we claim that $z = x'+y'$ for some $x' \preceq x$ and $y' \preceq y$.
Let $i \in \N$ be largest index such that $x_i=1$ or $y_i=1$. Say $x_i=1$.
If $z_i=0$, then clearly $z \prec x$, so we can take $x'=z$ and $y'=0$.
If $z_i=1$, then note that $(z-e_i) \prec (x-e_i) + y$.
By induction on $i$, obtain $(z-e_i) = x'' + y''$ for $x'' \preceq (x-e_i)$ and $y'' \preceq y$,
and choose $x' = x'' + e_i$ and $y' = y''$.
\end{proof}

The Hopf--Stiefel binary function $a \circ b$ can be defined on $\N \times \N$ as follows:
$$ a \circ b = |IS(a) + IS(b)|$$
Proposition~\ref{segs} can be restated as: $IS(a) + IS(b) = IS(a \circ b)$. 
This definition is relevant for us for the following reason. The cardinality of a sumset
of two sets of given cardinalities is minimized by taking the two sets to be initial segments:
$$ a \circ b = \min \left\{ |A+B| \;\bigg{|}\; A,B \in \Z_2^n,\; |A|=a,\; |B|=b \right\} $$
Note that here the sets are not required to be affinely generating.
This result can be deduced by the technique of compressions as we discuss below. See Lemma~\ref{sumcomp}.  

In particular, taking $A = IS(a)$ and $B = IS(b_1) \cup \left(e_n + IS(b_2)\right)$ for $n$ large enough,
one can verify the sub-distributive law:
$$ a \circ (b_1 + b_2) \leq a \circ b_1 + a \circ b_2 $$
Similarly, one can deduce the recursive relations for $a,b \leq 2^n$:
\begin{align*}
a \circ (2^n + b) = 2^n + a \circ b \\
(2^n + a) \circ (2^n + b) = 2^{n+1}
\end{align*}
These two formulas can be taken as an alternative definition of the Hopf--Stiefel function~\cite{pfister}.  

The function first arose in works of Hopf~\cite{hopf} and Stiefel~\cite{stiefel}.
They used tools from algebraic topology to prove that 
$a \circ b$ provides a lower bound for solutions of the Hurwitz problem, 
concerning real quadratic forms (see~\cite{shapiro}).
The relation to set addition in $\Z_2^n$ was given by Yuzvinsky~\cite{yuzvinsky}. As it
turns out, the Hopf--Stiefel function arises in the study of several more problems in various contexts.
There is also a base-$p$ analog of the this function for $p>2$, see~\cite{eliahou_kervaire_vector_spaces}. 
For a survey, see~\cite{eliahou_kervaire_hs}.

\subsection{Compressions}

For $I = \{i_1,i_2,...\} \subseteq [n]$, 
denote $H_I = \left\langle 0,e_{i_1},e_{i_2},...\right\rangle \triangleleft \Z_2^n$.
As usual, if $H$ is a subgroup of $G$, we denote by $G/H$ the collection of all $H$-cosets in $G$.
The \emph{$I$-compression} of a subset $A \subseteq \Z_2^n$ is defined by: 
$$C_I(A) = \bigcup_{T \in \Z_2^n/H_I} IS\left(|A \cap T|, T\right) $$
In words, in every $H_I$-coset $T$ we replace the elements of $A \cap T$
by a same-cardinality initial segment, with respect to the lexicographic order.
We say $A$ is compressed with respect to $I$, or \emph{$I$-compressed}, if $C_I(A)=A$.
In particular, lexicographic initial segments of $\Z_2^n$ are exactly all $[n]$-compressed sets. 

\begin{ex}\label{lost_affine}
$C_{\{1,2,3\}}(\{0,e_1,e_2,e_3,e_4\}) = \{0,e_1,e_2,e_1+e_2,e_4\}.$
\end{ex}

This notion of compression is closely related to 
the operation bearing the same name from extremal set theory (see, e.g.,~\cite{frankl_survey}). 
A subset of $\Z_2^n$ naturally corresponds to a family, a.k.a. set-system, $\cal F$ of subsets of $[n]$. 
We freely move between these terminologies if no confusion can occur.
An $\{i\}$-compression corresponds to the \emph{push-down operator} $T_i$, 
which replaces $J \in \cal F$ by $J \setminus \{i\}$ provided that $J \setminus \{i\} \not \in \cal F$.
If $\cal F$ is $\{i\}$-compressed for each $i$, 
then it is closed under taking subsets and is called a \emph{downset}.
The \emph{shift operator} $S_{ij}$ replaces $j$ by $i$ wherever possible. 
Namely, for every $J$ with $i,j \not \in J$ it replaces $J \cup \{j\}$ by $J \cup \{i\}$
given that the former belongs to $\cal F$ and the latter doesn't.
We say that $\cal F$ is \emph{shift-minimal} if it is invariant to all shifts $S_{ij}$ where $i<j$.
One can check that being $\{i,j\}$-compressed for all $i,j \in [n]$ 
corresponds to being a shift-minimal downset.

Compression can simplify matters substantially,
while preserving several useful features of the set-system.
Here are some observations about compressions. 
These and others are found in~\cite{green_tao}. The
proofs are straightforward, working coset by coset.

\begin{lemma}[Properties of compressions]\label{comp}
Suppose $A \subseteq \Z_2^n$ and $I \subseteq [n]$.
\begin{itemize}
\item[(1)] $|C_I(A)| = |A|$.
\item[(2)] $C_I(A)$ is $I$-compressed.
\item[(3)] $\hbar(C_I(A)) \leq \hbar(A)$ with equality iff $A$ is $I$-compressed.
\item[(4)] An $I$-compressed set is $J$-compressed for all $J \subseteq I$.
\item[(5)] $C_I(A) \subseteq C_I(B)$ for all $A \subseteq B$. \hfill \qed
\end{itemize}
\end{lemma}

Compressions behave well on sumsets.
By Proposition~\ref{segs}, one can deduce that the sum of two $I$-compressed subsets is $I$-compressed too.
The following well-known lemma deals with the compression of a sum of two general subsets.
For the sake of completeness, we prove it here, following~\cite{bollobas_leader} and~\cite{green_tao}.

\begin{lemma}[Sumset compression]\label{sumcomp}
Suppose $A,B \subseteq \Z_2^n$ and $I \subseteq [n]$. Then
$$ C_I(A)+C_I(B) \subseteq C_I(A+B) .$$
Consequently $|C_I(A)+C_I(B)| \leq |A+B|$.
\end{lemma}

\begin{proof}
We use a double induction, on $|I|$ and on $\hbar(A)+\hbar(B)$.
For the induction step, suppose that for some $J \subsetneq I$ either $A$ or $B$ is not $J$-compressed. In this case
$$ C_I(A) + C_I(B) = C_I(C_J(A)) + C_I(C_J(B)) \subseteq C_I(C_J(A) + C_J(B)) \subseteq C_I(C_J(A+B)) = C_I(A + B)$$
Both inclusions are by the induction hypothesis:
the first one since $\hbar(C_J(A))+\hbar(C_J(B)) < \hbar(A)+\hbar(B)$ by property (3) of Lemma~\ref{comp},
and the second one since $|J| < |I|$ and by property (5).
The equalities are by property (4).

It only remains to verify the lemma for $A$ and $B$ that are both $J$-compressed for all $J \subsetneq I$.
We start with the simpler case $n = |I|$.

What are the subsets of $G = \Z_2^n$ that are $J$-compressed for all $J \subsetneq [n]$? 
By property (4), all initial segments are such.
If $S \subseteq G$ is not an initial segment, 
then necessarily $x \notin S$ and $y \in S$ for some consecutive $x \prec y$.
The only consecutive pair in $G$ that is not contained in a proper $H_J$-coset is
$(e_1 + ... + e_{n-1}) \prec e_n$. 
One can verify, for example by $S$ being $[2...n]$-compressed, 
that the only such set is $S = H_{[n-1]} \setminus \{e_1 + ... + e_{n-1}\} \cup \{e_n\}$. 
In conclusion, it is enough to check the case where $A$ and $B$ are initial segments or equal to $S$.
Now there are four cases to consider:
\begin{enumerate}
\item 
If both $A$ and $B$ are initial segments, 
then by Proposition~\ref{segs} $A+B$ is an initial segment too. \\
$\Rightarrow C_I(A) + C_I(B) = A+B = C_I(A+B)$
\item 
If $A=B=S$, 
then note that $|S| \leq |S+S|$ and $C_I(S) = H_{[n-1]}$. \\
$\Rightarrow C_I(S) + C_I(S) = C_I(S) \subseteq C_I(S+S)$
\item 
If $B=S$ and $A$ is an initial segment with $|A| \leq |S|$, 
then $A = C_I(A) \subseteq C_I(S) = H_{[n-1]}$. \\
$\Rightarrow C_I(A) + C_I(S) = C_I(S) \subseteq C_I(A+S)$
\item
If $B=S$ and $A$ is an initial segment with $|A| > |S|$ then $|A|+|S| > |G|$. \\
This means $A+S=G$, as the reader may verify by a standard pigeonhole argument. \\
$\Rightarrow C_I(A) + C_I(S) = G = C_I(G) = C_I(A+S)$
\end{enumerate}
The case $n > |I|$ is implied by the case $n = |I|$:
\begin{align*}
C_I(A) + C_I(B) &= \bigcup\limits_{H_c \in G/H_I} \left(\left(C_I(A) + C_I(B)\right) \cap H_c\right) \\
&= \bigcup\limits_{H_c \in G/H_I} \bigcup\limits_{H_a+H_b=H_c} \left(\left(C_I(A) \cap H_a\right) + \left(C_I(B) \cap H_b\right)\right) \\
&= \bigcup\limits_{H_c \in G/H_I} \bigcup\limits_{H_a+H_b=H_c} \left(C_I\left(A \cap H_a\right) + C_I\left(B \cap H_b\right)\right) \\
&\subseteq \bigcup\limits_{H_c \in G/H_I} \bigcup\limits_{H_a+H_b=H_c} C_I\left(\left(A \cap H_a\right) + \left(B \cap H_b\right)\right) \\
&\subseteq \bigcup\limits_{H_c \in G/H_I} C_I\left(\bigcup\limits_{H_a+H_b=H_c} \left(\left(A \cap H_a\right) + \left(B \cap H_b\right)\right)\right) \\
&= \bigcup\limits_{H_c \in G/H_I} C_I\left((A+B) \cap H_c\right) \\
&= \bigcup\limits_{H_c \in G/H_I} \left(C_I(A+B) \cap H_c\right) \\
&= C_I(A+B) .
\end{align*}
The first and second inequalities are simply dividing into cases, 
according to the involved $H_I$-cosets.
The third one holds because compressions work coset-wise.
Then there is inclusion by the assumption on the case $I = [n]$,
applied to our $H_I$ and translated to the relevant $H_I$-cosets.
And then, inclusion of initial segments,
because the union is at least as large as each of its components.
The three remaining equalities are similar to the first three.
\end{proof}

\subsection{Compressions that Preserve Affine Generation}

As Lemma \ref{sumcomp} shows, in the problems we consider here, 
compressing the sets under consideration can only improve our objective function. 
However, we are restricting ourselves to affinely generating sets and compression may destroy this property
(e.g., Example \ref{lost_affine}). 
Therefore, our strategy is to keep compressing as long as affine generation is maintained. 
To this end we introduce the following definition.

Suppose that $A \supseteq E$, where $E = \{0,e_1,e_2,...,e_n\}$ is the standard affine basis of $\Z_2^n$. 
If $A$ is $I$-compressed for every $I$ such that $C_I(A) \supseteq E$, 
we say that $A$ is $\langle\langle E \rangle\rangle$-\emph{compressed}.
Note that by part (3) of Lemma~\ref{comp} every set $A$ containing $E$, 
can be turned into an $\langle\langle E \rangle\rangle$-compressed set by a finite sequence of such compressions.
It turns out that $\langle\langle E \rangle\rangle$-compressed sets are very structured.

\begin{lemma}[Structure of $\langle\langle E \rangle\rangle$-compressed sets]\label{structure}
Let $A \subseteq \Z_2^n$ be an $\langle\langle E \rangle\rangle$-compressed set.
\begin{itemize}
\item[(1)] 
$A$ is a shift-minimal downset.
\item[(2)] 
$A$ contains a subgroup of maximal size $H \triangleleft \Z_2^n$ 
of the form $H = \left\langle 0,e_1,...,e_h \right\rangle$.
\item[(3)] 
$A$ is $\{1,...,h,h+i\}$-compressed for every $1 \leq i \leq m = \mathrm{codim}\;H$.
\item[(4)] 
$A \subseteq H+E$, i.e. $A = H \cup A_1 \cup A_2 \cup ... \cup A_m$ where $A_i = A \cap (e_{h+i} + H)$.
\item[(5)] 
For $1 \leq i \leq m$, $0 < |A_i| < |H|$.
\item[(6)] 
For $1 \leq i < j \leq m$, $|A_i| + |A_j| \leq |H|$.
\item[(7)] 
If $m > 1$, then $|A| \leq \left(1 + \frac{m}{2}\right)|H|$.
\end{itemize}
\end{lemma}

\begin{proof} 
The proofs are fairly straightforward.
\begin{itemize}
\item[(1)]
It is a simple observation that both $\{i\}$-compressions and $\{i,j\}$-compressions preserve $E \subseteq A$.
Hence $A$ must already be compressed with respect to these sets, i.e., a shift-minimal downset.
\item[(2)]
Let $h$ be the maximal dimension of a subgroup contained in $A$. As shown below in Lemma~\ref{at_least_h},
a subgroup of dimension $h$ must contain an element of Hamming weight at least $h$. 
By shift-minimality $e_1 + e_2 + ... + e_h \in A$, and by the downset property
$H = \left\langle 0,e_1,...,e_h \right\rangle \subseteq A$.
\item[(3)]
Denote $I = \{1,...,h,h+i\}$.
The sets $H \cup \{e_{h+i}\}$ and $\{e_{h+j}\}$ for $j \neq i$ are initial segments of their $H_I$-cosets.
These sets cover $E$ and remain included in $A$ through the $I$-compression.
\item[(4)] 
By the downset property, it is sufficient to show $e_{h+i} + e_{h+j} \notin A$ for each $1 \leq i < j \leq m$.
Indeed, if $A$ contains $e_{h+i} + e_{h+j}$ then it contains $e_{h+i} + H$ by being $\{1,...,h,h+j\}$-compressed.
This implies $H \cup (e_{h+i}+H) \subseteq A$,
contrary to the maximality of the subgroup $H$ in $A$.
\item[(5)] 
For the lower bound note that $e_{h+i} \in E \subseteq A$.
On the other hand, if $|A_i| = |H|$ then $H \cup (e_{h+i}+H) \subseteq A$, contrary, again, to the maximality of $H$.
\item[(6)]
Note that $e_{h+j} \in A$,
while some lexicographically smaller elements in $e_{h+i} + H$ are not contained in $A$.
Therefore $A$ can't be $I$-compressed for $I = \{1,...,h,h+i,h+j\}$.
Since it is $\langle\langle E \rangle\rangle$-compressed, this means $e_{h+j} \notin C_I(A)$.
Equivalently, $|A \cap H_I| \leq 2|H|$, which leads to our claim.
\item[(7)]
If $|A_i| \leq \frac{1}{2}|H|$ for every $i$, 
clearly $|A|  = |H| + \sum_{i=1}^m|A_i| \leq \left(1 + \frac{m}{2}\right)|H|$.
Otherwise, $|A_i| > \frac{1}{2}|H|$ for some $i$, 
thus $|A_j| \leq |H| - |A_i| < \frac{1}{2}|H|$ for every $j \neq i$.
So $|A_i| + |A_j| \leq |H|$ for some $i$ and $j$,
and the remaining $A_j$'s are no bigger than $\frac{1}{2}|H|$.
\end{itemize}
\end{proof}

\begin{lemma}\label{at_least_h}
Let $H$ be an $h$-dimensional subgroup of $\Z_2^n$. 
Then $H$ contains an element of Hamming weight at least $h$.
\end{lemma}

\begin{proof}
If $h=n$, take $e_1 + e_2 + ... + e_n$.
Otherwise, there exists a basis element $e_i$ such that $e_i \notin H$.
In this case, 
moving from $H$ to $C_{\{i\}}(H)$ simply deletes $e_i$ from the standard basis representations of $H$'s elements,
thereby not increasing their Hamming weights.
Now note that $C_{\{i\}}(H)$ is an $h$-dimensional subgroup of 
$\left\langle 0, e_1, ..., e_{i-1}, e_{i+1}, ..., e_n\right\rangle$,
and by induction on $n$ contains an element of Hamming weight at least $h$.
\end{proof}

\section{The Freiman--Ruzsa Theorem in \texorpdfstring{$\Z_2^n$}{}}\label{freisect}

For $A \subseteq \Z_2^n$ we refer to $|\left\langle A\right\rangle|/|A|$ as $A$'s \emph{spanning constant}
and to $K=|A+A|/|A|$ as its \emph{doubling constant}.
The Freiman--Ruzsa theorem gives an upper bound on the spanning constant in terms of $K$.
We first review the theorem and some of its quantitative aspects.
Then we calculate the bound explicitly, 
and in particular we determine its correct asymptotics which turns out to be $\Theta (2^{2K}/K)$.
We present the proof in two stages, starting with the asymptotic estimates. 
We find this presentation convenient, since the proof of the asymptotic bound already contains our main ideas.

\subsection{Brief Review of the Freiman--Ruzsa Theorem}

Freiman's celebrated theorem~\cite{freiman} states that 
if $A \subset \Z$ is a finite subset with $|A+A| \leq  K|A|$,
then $A$ is included in a generalized arithmetic progression, 
whose size (relative to $|A|$) and dimension are bounded.
The bounds depend only on $K$ and not on $|A|$.
Ruzsa~\cite{ruzsa_proof_freiman1,ruzsa_proof_freiman2} 
has made crucial contributions to this area. More recently much work
was done on similar problems where $\Z$ is replaced by other groups.
In particular Ruzsa~\cite{ruzsa_thm} proved the analogous result for abelian torsion groups.
See~\cite{viola} for a nice exposition.

\begin{thm}[Ruzsa]
Let $G$ be an abelian group in which every element has order at most $r$.
If $A$ is a finite subset of $G$ with $|A+A| \leq  K|A|$, 
then $A$ is contained in a coset of a subgroup $H \triangleleft G$ of size $|H| \leq f(r,K)|A|$,
where 
$$f(r,K) \leq K^2 r^{K^4}.$$
\end{thm}

Better estimates on $f(r,K)$ were subsequently found. We denote by $F(r,K)$
the smallest bound for which this statement holds.
Note that $F(r,K)$ is non-decreasing in $K$ and $F(r,1)=1$.

By considering the case where $A$ is an affine basis of $\Z_{r}^{2(K-1)}$
we see that $F(r,K) \geq r^{2K-O(\log K)}$ (see Example~\ref{IPE} below).
This suggests the following conjecture~\cite{ruzsa_thm}.

\begin{conj}[Ruzsa]\label{conjexp}
For some $C \geq 2$ we have $F(r,K) \leq r^{CK}$.
\end{conj}

In an attempt to understand the role of torsion in these phenomena,
much work was dedicated to the special case $r=2$, where $G=\Z_2^n$.
This work is also motivated by the role that $\Z_2^n$ plays in discrete
mathematics and in particular in coding theory~\cite{zemor_codes}.
We introduce the following notation:
$$ F(K) = F(2,K) = \sup\left\{ \frac{\left|\left\langle A\right\rangle\right|}{|A|} \;\Bigg{|}\; 
              A \subseteq \Z_2^n, \; n \in \N, \; \frac{|A+A|}{|A|} \leq  K \right\} $$
As already observed by Ruzsa~\cite{deshouillers_hennecart_plagne}, for $r=2$ his method gives
somewhat more, namely $F(K) \leq K2^{\left\lfloor K\right\rfloor^3-1}$.
Later work by Green and Ruzsa~\cite{green_ruzsa} gave $F(r,K) \leq K^2 r^{2K^2-2}$,
which was again refined for $r=2$ to $F(K) \leq 2^{O\left(K^{3/2}\log K\right)}$ by Sanders~\cite{sanders}.
Using compressions, Green and Tao~\cite{green_tao} were able to prove 
$F(K) \leq 2^{2K + O\left(\sqrt{K} \log K\right)}$.
Note that this confirms Conjecture~\ref{conjexp} for $r=2$.
The best bound so far is due to Konyagin~\cite{konyagin} 
who further improved this method to derive $F(K) \leq 2^{2K + O\left(\log K\right)}$.

The range of small $K$ has received some attention as well. 
In the sub-critical range $K<2$, the exact value of $F(K)$ is known to be
$F(K) = K$ for $1 \leq K < 7/4$ and $F(K) = \frac87 K$ for $7/4 \leq K < 2$.
See~\cite{diao,green_tao,hennecart_plagne,lev_structure,zemor}. 
For $K \leq 12/5$ we have $F(K) \leq (2K-1)/(3K-K^2-1)$ and for $12/5 < K < 4$, a recursive formula is available.
See~\cite{deshouillers_hennecart_plagne}.

The following simple construction~\cite{ruzsa_thm} provides a lower bound on $F(K)$.

\begin{ex}[Independent Points]\label{IPE}
Consider the subset:
$$A_{[t]} = \{0,e_1,e_2,...,e_t\} \subseteq \Z_2^t$$
Here, for $t \in \N$ we have $ F\left(\frac{\binom{t}{2} + t + 1}{t + 1}\right) \geq \frac{2^t}{t + 1} $,
and by monotonicity one can obtain:
$$F(K) \geq \frac{1}{4K}2^{2K}(1-o(1))$$
\end{ex}
              
\subsection{Asymptotics of \texorpdfstring{$F(K)$}{}}

We first prove a new upper bound, 
which coincides with the construction in Example~\ref{IPE} for $t \in \N$.

\begin{thm}\label{newbound}
$F \left(\frac{\binom{t}{2} + t + 1}{t+1}\right) \leq \frac{2^t}{t+1}$ holds for $2 \leq t \in \R$. Consequently,
$$ F( K) \leq \frac{1}{2K}2^{2K}(1-o(1)) .$$
\end{thm}

The exponential term $2^{2 K}$ is as in~\cite{green_tao,konyagin},
but the polynomial coefficient $1/K$ is new.
Thus it re-proves Conjecture~\ref{conjexp} for $r=2$ with $C=2$.
This bound and Example~\ref{IPE} determine the asymptotics of $F(K)$ up to a factor of $2$.
In the next section we calculate $F(K)$ exactly, 
and show that the gap is unavoidable and results from the oscillations in $F(K)$.

\begin{proof}
For an affinely generating subset $A \subset G = \Z_2^n$, it is sufficient to prove: 
\begin{equation}\label{upper1}
|A| = \frac{t+1}{2^t}|G| \;\;\;\; \Rightarrow \;\;\;\; |A+A| \geq \frac{\binom{t}{2}+t+1}{2^t}|G|
\end{equation}
where $2 \leq t \in \R$. Since both expressions are monotone in $t$, the theorem follows.

As in~\cite{green_tao}, the main tool is reduction to compressed sets of some sort.
First, since $\left\langle A\right\rangle = G$ we can assume
that $A$ contains an affine basis for $G$.
But $|A|,|A+A|$ are not affected by invertible affine transformations,
so we may assume without loss of generality $E \subseteq A$, 
where $E = \{0,e_1,e_2,...,e_n\}$ is the standard affine basis of $G$.
Now we assume without loss of generality that $A$ is $\langle\langle E \rangle\rangle$-compressed.
Indeed, supposing (\ref{upper1}) holds for $\langle\langle E \rangle\rangle$-compressed subsets,
we proceed to general subsets inducting on $\hbar(A)$.
Let $I \subseteq [n]$ be a set such that $E \subseteq C_I(A) \neq A$.
By Lemma~\ref{sumcomp}, $|C_I(A)+C_I(A)| \leq |A+A|$ while $|C_I(A)| = |A|$, 
so $A$ satisfies (\ref{upper1}) provided that $C_I(A)$ does.
The inductive argument applies, since $\hbar(A) > \hbar(C_I(A))$ by Lemma~\ref{comp}(3).

We continue the proof using the structure of $\left\langle \left\langle E \right\rangle\right\rangle$-compressed sets.
As in Lemma~\ref{structure} let $H \subseteq A$ be a maximal subgroup, 
$h = \dim H$, $m = \mathrm{codim}\;H$ and $A_i = A \cap (e_{h+i} + H)$ for $1 \leq i \leq m$.
By Lemma~\ref{structure}(7), $|A| \leq (1+m/2)|H|$, and an upper bound on $m$ is given by
$$ \frac{1+\frac{m}{2}}{2^m} \geq \frac{|A|}{|G|}  $$
where the case $m=1$ follows from the assumption $2 \leq t$.

Given $m$, Lemma~\ref{structure}(4) gives a decomposition of $A$ into $m+1$ parts,
and we use it to show that $A+A$ is at least $\sim m/2$ times larger than $A$.
This is shown by the following calculation, 
where all indices go from $1$ to $m$ and all unions are disjoint:
\begin{align*}
A \;&=\; H \;\cup\; \bigcup_i A_i \\
\Rightarrow\;\;\;\;\;\; A+A \;&=\; H \;\cup\; \bigcup\limits_i(A_i + H) \;\cup\; \bigcup\limits_{i<j}(A_i+A_j) \\
\sum\limits_{i<j}|A_i+A_j| \;&\geq\; \sum\limits_{i<j}\max \left(|A_i|,|A_j|\right)
\;\geq\; \sum\limits_{i<j}\frac{|A_i|+|A_j|}{2} \;=\; \frac{m-1}{2} \sum\limits_{i}|A_i| 
\;=\; \frac{m-1}{2} \left(|A| - |H|\right) \\
\Rightarrow\;\;\;\; |A+A| \;&\geq\; |H| + m|H| + \frac{m-1}{2} \left(|A| - |H|\right) 
\;=\; \frac{m+3}{2} \cdot \frac{|G|}{2^m} + \frac{m-1}{2} |A|
\end{align*}
The right-hand side is decreasing in $m$ in the real interval where $((m+3)\log 2 - 1)/{2^m} > |A|/|G|$.
This interval includes the range of our interest, which is $(m/2 + 1)/{2^m} \ge |A|/|G| = (t+1)/2^t$, or equivalently $m \le t-1$.
Thus, we obtain a lower bound on $|A+A|$ by evaluating this expression at $t-1$, namely:
$$ |A+A| \geq \frac{(t-1)+3}{2^{(t-1)+1}} |G| + \frac{(t-1)-1}{2} \cdot \frac{t + 1}{2^t} |G| =
\frac{\binom{t}{2}+t+1}{2^{t}}|G| $$
\end{proof}

\subsection{Exact Calculation of \texorpdfstring{$F(K)$}{}}

Theorem~\ref{formulaF}, which we will shortly prove, provides an explicit formula of $F(K)$.
This enables one to rederive the asymptotics of $F(K)$, and to deduce the following corollary.

\begin{cor}
Both bounds in the asymptotic inequalities
$$ \frac{1}{4 K}2^{2 K}(1-o(1)) \;\leq\; F(K) \;\leq\; \frac{1}{2 K}2^{2 K}(1-o(1)) $$
are sharp up to the $o(1)$ terms. \hfill \qed
\end{cor}

It also settles the following conjecture of Diao~\cite{diao}.

\begin{cor}
$F(K)$ is a piecewise linear function. \hfill \qed
\end{cor}

\begin{figure}[hbt]
\begin{center}
\includegraphics[width=1.00\textwidth,natwidth=2445,natheight=1845]{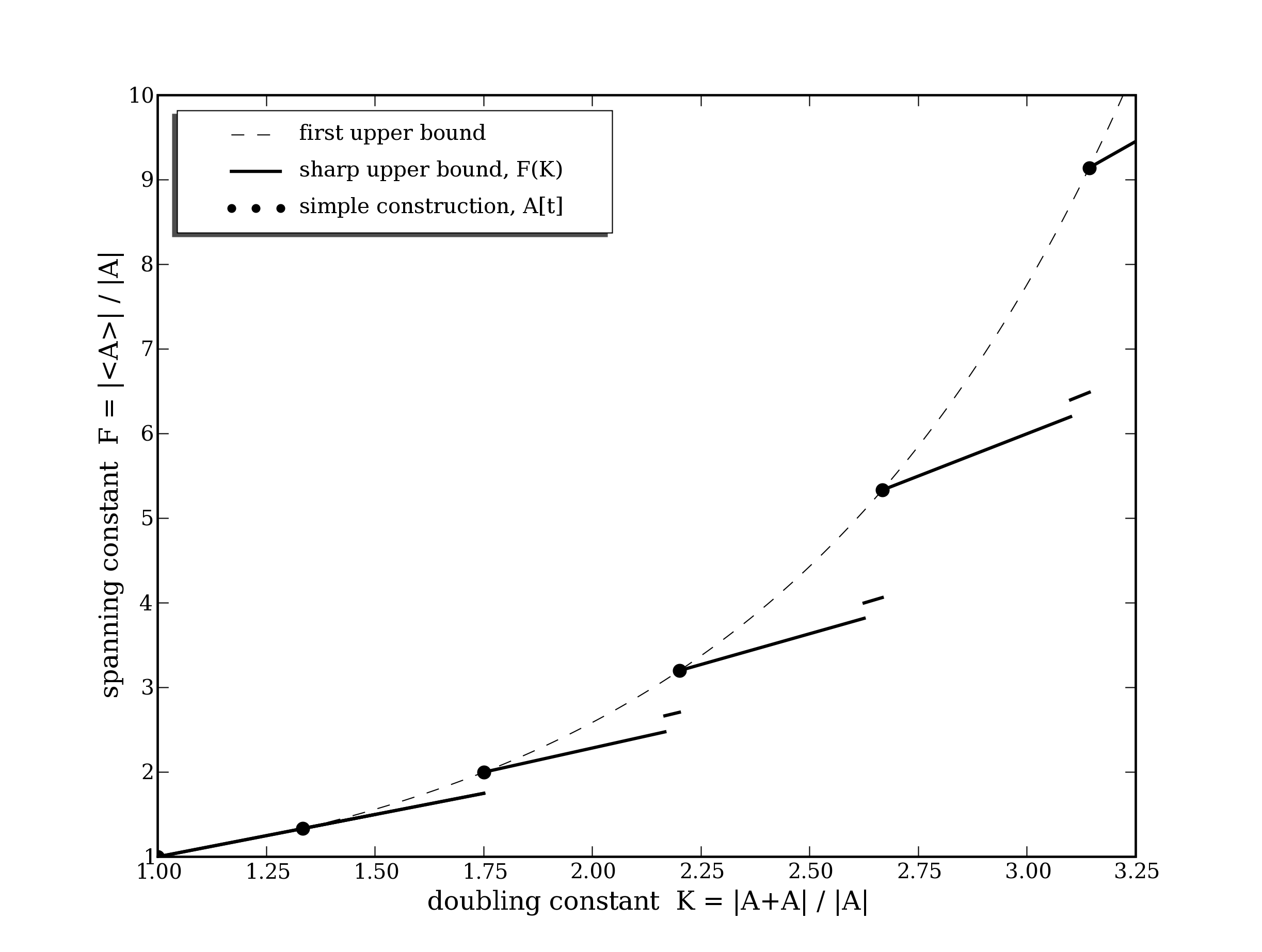}
\end{center}
\caption{An illustration of $F(K)$}
\label{fig_f}
\end{figure}

In order to calculate $F(K)$, it is useful to consider a related function $\tilde{K}(\tilde{F})$,
which is defined for rational numbers of the form $\tilde{F} = {2^a}/{b} \geq 1$.
$$ \tilde{K}(\tilde{F}) = \inf\left\{ \frac{|A+A|}{|A|} \;\Bigg{|}\; 
A \subseteq \Z_2^n, \; n \in \N, \; \frac{\left|\left\langle A\right\rangle\right|}{|A|} = \tilde{F} \right\}$$
That is, the minimal doubling constant of an affinely generating set of relative size exactly $1/\tilde{F}$.
By definition $ F(K) = \sup \{ \tilde{F} \;|\; \tilde{K}(\tilde{F}) \leq K \}$. 
Theorem~\ref{newbound} asserts $\tilde{K}\left(2^t/(t+1)\right) \geq \left(\binom{t}{2}+t+1\right)/(t+1)$
for real $t \geq 2$, and by Example~\ref{IPE} it is an equality for $t \in \N$.
In order to analyze $\tilde{K}(\tilde{F})$, we refine the arguments in the proof of Theorem~\ref{newbound},
and elaborate on the construction in Example~\ref{IPE}. 
This yields a better view of the structure of sets with a small doubling constant.
We begin by describing the extended example.

\begin{ex}\label{IPE2}
For non-negative integers $s,t$ such that $s < t$, consider the subset:
$$A_{[t,s]} = \{0,e_0,e_1,e_2,...,e_t,e_0+e_1,e_0+e_2,...,e_0+e_{t-s}\} \subseteq \Z_2^{t+1}$$
It is not hard to verify that
$$|A_{[t,s]}| = 2(t+1)-s, \;\;\;\;\;\;\;\; 
|A_{[t,s]}+A_{[t,s]}| = 2\left(\binom{t}{2}+t+1\right) - \binom{s}{2}, \;\;\;\;\;\;\;\;
\left|\left\langle A_{[t,s]}\right\rangle\right| = 2^{t+1}. $$
Therefore:
$$ \tilde{K}\left(\frac{2^t}{t+1-s/2}\right) \leq \frac{\binom{t}{2}+t+1-\binom{s}{2}/2}{t+1-s/2} $$
\end{ex}

This example provides an upper bound on $\tilde{K}(\tilde{F})$ for a discrete sequence of values. 
When $s=0$ it reduces to Example~\ref{IPE}.
However, $\tilde{K}(\tilde{F})$ is not necessarily monotone, 
so we cannot imitate the conclusion of Example~\ref{IPE} and extend the upper bound to general $\tilde{F}$.
Still, the following argument does the work.

\begin{lemma}[Sublinearity of $\tilde{K}(\tilde{F})$]\label{sublinearK}
If $F_1 < F_2$ are in $\tilde{K}$'s domain, then $ \frac{\tilde{K}(F_1)}{F_1} \geq \frac{\tilde{K}(F_2)}{F_2} $.
\end{lemma}

\begin{proof}
Let $F_2 = 2^a/b$ for some $a,b \in \N$.
Suppose $A_1 \subseteq \Z_2^n$ is an affinely generating set of size $|A_1| = 2^n/F_1$.
Let $m \in \N$ be large enough such that $a \leq n+m < b2^{n+m-a}$.
Consider $A_1' = A_1 \times \Z_2^m$, and note that $A_1'$ affinely generates $\Z_2^{n+m}$ and $|A_1'| = 2^{n+m}/F_1$.
Since $F_1 < F_2$ 
one can take a subset $A_2 \subseteq A_1'$ of cardinality $|A_2| = b2^{n+m-a} = 2^{n+m}/F_2$.
Moreover, by $m$'s choice $n+m+1 \leq |A_2|$, so a subset $A_2$ which affinely generates $\Z_2^{n+m}$ can be chosen.
Now from $A_2+A_2 \subseteq A_1'+A_1' = (A_1 + A_1) \times \Z_2^m$,
$$ \frac{|A_1+A_1|}{|A_1|} \cdot \frac1{F_1} = 
\frac{|A_1+A_1|}{2^n} = 
\frac{|A_1'+A_1'|}{2^{n+m}} \geq
\frac{|A_2+A_2|}{2^{n+m}} = 
\frac{|A_2+A_2|}{|A_2|} \cdot \frac1{F_2} \geq 
\frac{\tilde{K}(F_2)}{F_2}.$$
The task is accomplished by taking the infimum over $A_1$.
\end{proof}

\begin{cor}[Superlinearity of $F(K)$]\label{superlinearF}
$ \frac{F(K_1)}{K_1} \leq \frac{F(K_2)}{K_2}$ for every $1 \leq K_1 < K_2$. \hfill \qed
\end{cor} 

Example~\ref{IPE2} and Lemma~\ref{sublinearK} supply an upper bound on $\tilde{K}(\tilde{F})$.
The following lemma essentially claims that this bound is sharp.

\begin{lemma}[Formula for $\tilde{K}(\tilde{F})$]\label{formulaK}
Let $\tilde{F} \geq 1$ be of the form ${2^a}/{b}$ where $a,b \in \N$,
and let $s < t$ be the unique pair of non-negative integers for which
$$ \frac{2^t}{t+1-s/2} \leq \tilde{F} < \frac{2^t}{t+1-(s+1)/2} $$
Then
$$ \tilde{K}(\tilde{F}) = \frac{\binom{t}{2}+t+1-\frac12\binom{s}{2}}{2^t} \cdot \tilde{F} $$
\end{lemma}

\begin{figure}[thb]
\begin{center}
\includegraphics[width=1.00\textwidth,natwidth=2445,natheight=1845]{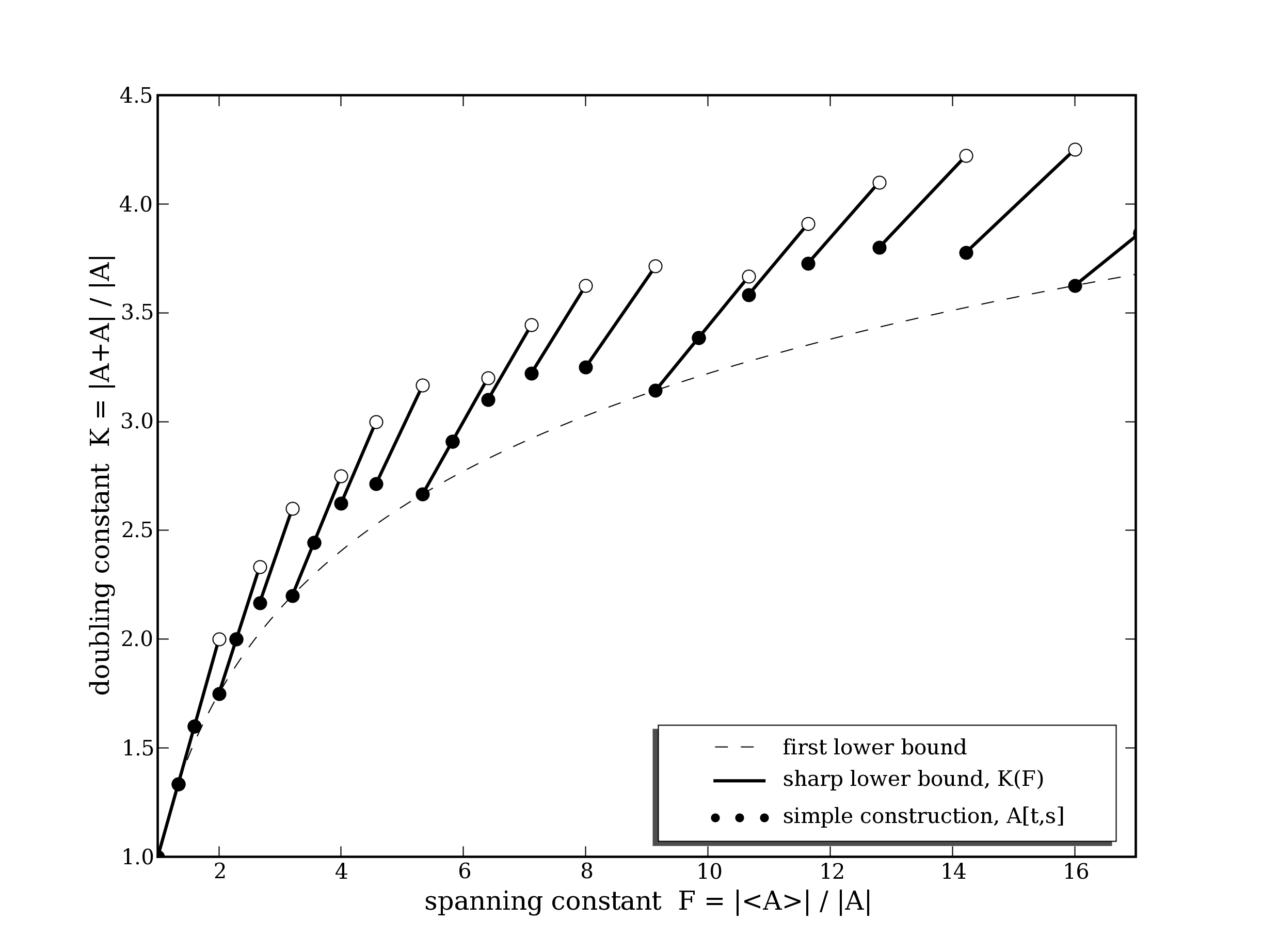}
\end{center}
\caption{An illustration of $\tilde{K}(\tilde{F})$}
\label{fig_k}
\end{figure}

Since the function $F(K)$ is basically the inverse of $\tilde{K}(\tilde{F})$,
Theorem~\ref{formulaF} is a direct consequence of Lemma~\ref{formulaK}.
Indeed, Figure~\ref{fig_f} is obtained by transposing the graph in Figure~\ref{fig_k},
and taking the maximum wherever the result is multivalued.
We omit further details.

One can notice that $\tilde{K}(\tilde{F})$ has a more complex structure than $F(K)$. 
Since Theorem~\ref{formulaF} employs the information in Lemma~\ref{formulaK} only partially, 
there may be a quicker way of calculating $F(K)$.
Nevertheless, we feel that the detailed description of $\tilde{K}(\tilde{F})$ 
is interesting in its own right, and may shed light on the non-trivial form of $F(K)$.

The proof of Lemma~\ref{formulaK} pursues the analysis in Theorem~\ref{newbound}'s proof,
involving more reduction steps which preserve $|A|$ without increasing $|A+A|$.
Through these reductions the structure of $A$ becomes similar to Example~\ref{IPE2}, 
so that its doubling constant can be calculated explicitly. 
We start with two reductions which can be formulated separately in terms of integer partitions.
All of the following will be motivated and applied later, in the proof of the lemma.

A non-increasing sequence of positive integers $a_1...a_m$ is 
an integer \emph{partition} of $a=\sum_i a_i$ into $m$ parts, and for short an $m$-\emph{partition} of $a$.
Recall the Hopf--Stiefel function $a \circ b$ from Section~\ref{toolsect}.
We are interested in the minimum of $\sum_{1 \leq i<j \leq m} a_i \circ a_j$ over all $m$-partitions of $a$.

A partition $a_1...a_m$ of $a$ is called \emph{compressed} if 
$a_i+a_j > 2^k \Rightarrow a_i \geq 2^k$ for each $k$ and $i<j$.
It will be implicit in the proof of Lemma~\ref{formulaK}, 
that at least one of the partitions that minimize $\sum_{i<j} a_i \circ a_j$ is compressed.
Here we shall restrict the discussion to compressed partitions.

A partition is called \emph{quasi-dyadic} if $a_1...a_{m-1}$ are powers of $2$.
No requirement is made on $a_m$. 
Note that a quasi-dyadic partition is always compressed.
Our first reduction basically asserts that 
the minimum of $\sum_{i<j}a_i \circ a_j$ is attained by a quasi-dyadic partition.

\begin{lemma}\label{quasidyadic}
A compressed $m$-partition of $a$ that minimizes $\sum_{i<j} a_i \circ a_j$ is quasi-dyadic.
\end{lemma}

\begin{proof}
Otherwise, consider the smallest $1 \leq i < m$ for which $a_i$ is not a power of $2$,
say $2^k < a_i < 2^{k+1}$. Since the partition is compressed, $a_i + a_{i+1} \leq 2^{k+1}$.
We 'transfer mass' from $a_i$ to $a_{i+1}$.
Replace $a_i$ by $a'_i = 2^k$, and replace $a_{i+1}$ by $a'_{i+1} = a_{i+1} + a_i - 2^k$. 
Note that $a'_i \geq a'_{i+1}$ and monotonicity is preserved.

How does this move affect $\sum_{i<j} a_i \circ a_j$?
By the choice of $i$, for $j < i$ we have $a_j = 2^l$, 
where $l > k$ as a partition is non-increasing. 
Thus the terms involving $a_j$ are unchanged:
$$a_j \circ a_i + a_j \circ a_{i+1} = 2^l + 2^l = a_j \circ a'_i + a_j \circ a'_{i+1}.$$
For $j > i+1$ we know $a_j \leq a_{i+1} < 2^k$. 
By the recursive definition of the Hopf--Stiefel function, and the sub-distributive law:
\begin{align*} 
a_j \circ a_i + a_j \circ a_{i+1} &= 
2^k + a_j \circ (a_i-2^k) + a_j \circ a_{i+1} \\
&\geq 2^k + a_j \circ (a_i-2^k + a_{i+1})
= a_j \circ a'_i + a_j \circ a'_{i+1}.
\end{align*}
Finally, again by the recursive definition the mixed term becomes strictly smaller:
$$ a_i \circ a_{i+1} = 2^k + (a_i-2^k) \circ a_{i+1} > 2^k = a'_i \circ a'_{i+1} $$ 
The combination of the last three calculations yields that the sum $\sum_{i<j} a_i \circ a_j$
can be made smaller by changing the partition, in contradiction to the minimality assumption.
\end{proof}

Since $2^k \circ a = 2^k$ for $a \leq 2^k$, in the quasi-dyadic case the summation can be simplified :
$$ \sum\limits_{1 \leq i<j \leq m}a_i \circ a_j \;\;= 
\sum\limits_{1 \leq i<j \leq m}\max(a_i,a_j) \;\;= 
\sum\limits_{1 \leq i<j \leq m} a_i \;\;= 
\sum\limits_{i=1}^{m} (m-i) \cdot a_i $$

It is natural to conjecture that the minimum is obtained when $a_1...a_m$ are 'almost' equal.
A quasi-dyadic $m$-partition of $a$ is \emph{quasi-fair}
if for some $k \in \N$, $a_i \in \left\{2^k,2^{k-1}\right\}$ for each $1 \leq i \leq m - 1$.
For example $4+4+2+2+2+1$ and $4+4+4+3$ and $8+4+3$ are some quasi-fair partitions of $15$.
The following properties of quasi-fair partitions are easily verified:
\begin{enumerate}
\item In the above definition one can choose 
$$k = \left\lceil \log_2(a/m)\right\rceil,$$ 
and then exactly $a_1...a_j$ exceed $2^{k-1}$ where 
$$j = \left\lceil a/2^{k-1} \right\rceil - m.$$
\item For every two positive integers $m \leq a$, 
there exists a unique quasi-fair $m$-partition of $a$.
\item If $a_1...a_m$ and $a_1'...a_m'$ are the quasi-fair $m$-partitions of $a \leq a'$, 
then $a_i \leq a_i'$ for all $i$.
\item A sub-partition (in the sense of a sub-sequence) of a quasi-fair partition is quasi-fair.
\end{enumerate}
Now we are ready to state the second reduction.
\begin{lemma}\label{quasifair}
The minimum of $\sum_{i<j} a_i \circ a_j$ over all quasi-dyadic $m$-partitions of $a$
is obtained only by the quasi-fair one.
\end{lemma}
\begin{proof}
This lemma can be verified by induction on $m$.
For a partition that minimizes the sum, 
it is enough to show $a_1 = 2^k$ for $k = \left\lceil \log_2(a/m)\right\rceil$.
By the induction hypothesis $a_2...a_m$ are quasi-fair,
and thus constitute the unique quasi-fair sub-partition we are looking for.
By the monotonicity property applied on $a_2...a_m$, 
for a competing sequence $a'_1...a'_m$ with $a'_1 > a_1$,
necessarily $a'_i \leq a_i$ for $i \geq 2$, and consequently:
\begin{align*} 
\sum\limits_{i=1}^{m} (m-i) \cdot a_i &= 
\sum\limits_{i=1}^{m} (m-i) \cdot a_i + (m - 1) \left(\sum\limits_{i=1}^{m}a'_i - 
\sum\limits_{i=1}^{m}a_i\right) \\
&= \sum\limits_{i=1}^{m} (m-i) \cdot a'_i - \sum\limits_{i=2}^{m}(i-1) (a_i - a'_i) < 
\sum\limits_{i=1}^{m} (m-i) \cdot a'_i 
\end{align*}
\end{proof}

With these reductions in hand, we can complete the calculation of $\tilde{K}(\tilde{F})$.

\begin{proof}(of Lemma~\ref{formulaK})
Let $\left\langle A \right\rangle = G = \Z_2^n$. 
Lemma~\ref{formulaK} is proved by showing the following lower bound on $|A+A|$, 
which is reached by Example~\ref{IPE2} and Lemma~\ref{sublinearK}:
\begin{equation}\label{setting} 
\frac{t+1-(s+1)/2}{2^t} < \frac{|A|}{|G|} \leq \frac{t+1-s/2}{2^t} \;\;\;\; \Rightarrow \;\;\;\;
\frac{|A+A|}{|G|} \geq \frac{\binom{t}{2}+t+1-\frac12\binom{s}{2}}{2^t}
\end{equation}
If $|A| > \frac12 |G|$, then by the pigeonhole principle $A+A=G$, as required in the cases $t=1,2$.
Hence we may assume $|A| \leq \frac12 |G|$ and $t \geq 3$.

We start as in Theorem~\ref{newbound}.
We first assume without loss of generality that $A$ is $\langle\langle E \rangle\rangle$-compressed, 
and therefore by Lemma~\ref{structure} has the following properties:

\begin{itemize}
\item
There exists a subgroup 
$H = \left\langle 0,e_1,...,e_h \right\rangle$ such that $A = H \cup A_1 \cup A_2 \cup ... \cup A_m$,
where $A_i = A \cap (e_{h+i} + H)$ and $m = \mathrm{codim}\;H$.
\item
$1/2^m < |A|/|G| \leq (1+\frac{m}{2})/2^m$.
By the assumptions $\frac{t+2}{2^{t+1}} < |A|/|G| \leq \frac12$, we can write $1 < m < t$.
\item 
Each $A_i$ is a lexicographic initial segment of $e_i + H$.
Therefore $A$ is uniquely determined by the sequence $a_1,...,a_m$ where $a_i = |A_i|$. 
Note that $0 < a_i < 2^h$.
\item
By shift-minimality $a_1 \geq a_2 \geq ... \geq a_m$.
In other words, $a_1,...,a_m$ is a partition of $a = |A| - {|G|}/{2^m}$.
\end{itemize}

As in Theorem~\ref{newbound},
we use these properties to write $A+A$ as a disjoint union of its intersections with $H$-cosets, 
which are of three forms: $H$, $H+A_i$ and $A_i+A_j$.
Since the $A_i$'s are initial segments of their cosets, 
the sumsets of the third form can be expressed via the Hopf--Stiefel function:
$$ |A+A| = |H| + m|H| + \sum\limits_{1 \leq i<j \leq m}|A_i+A_j| = 
\frac{m+1}{2^m}\cdot|G| + \sum\limits_{1 \leq i<j \leq m}a_i \circ a_j $$
This equation makes it interesting to find partitions $a_1...a_m$ of $a$, 
that minimize $\sum_{i<j} a_i \circ a_j$.

We next show that the partition $a_1...a_m$ is compressed.
For $i<j$ and $1 \leq k \leq h$ we exclude the case where $a_i < 2^k < a_i+a_j$
by the assumption that $A$ is already $\langle\langle E \rangle\rangle$-compressed.
For $I = \{1,2,...,k,h+i,h+j\}$, let's examine the set $C_I(A)$.
$A_i$ is replaced by an initial segment of $e_{h+i}+H$ of size $2^k$,
and $A_j$ is replaced by an initial segment of $e_{h+j}+H$ of size $a_i + a_j - 2^k$, 
which is not empty by assumption.
In other words, $a_i$ becomes $2^k$, and $E \subseteq A$ is preserved.

By Lemmas~\ref{quasidyadic}-\ref{quasifair}, if $|A+A|$ is minimal
then $a_1...a_m$ is the quasi-fair quasi-dyadic $m$-partition of $a = |A| - {|G|}/{2^m}$.
In this situation
$$
|A+A| - \frac{m+1}{2^m}\cdot|G| = 
\sum\limits_{i=1}^j(m-i) \cdot 2^k + \sum\limits_{i=j+1}^{m}(m-i) \cdot 2^{k-1} = 
\left[\binom{m}{2} - \frac12 \binom{m-j}{2}\right]2^k
$$
for $0 \leq j \leq m$ and $0 < k < (\dim G - m)$ such that 
$$ \frac{m+j-1}{2} \cdot 2^k \;<\; |A|-\frac{|G|}{2^m} \;\leq\; \frac{m+j}{2} \cdot 2^k .$$

\begin{rmk*}
Note that in the cases $j=m-1$ or $j=m$ we can choose $j'=0$ and $k'=k+1$ as well.
We could avoid this freedom of choice by not permitting $j=0$, 
but since it does not affect the resulting $|A+A|$, we allow both ways.
\end{rmk*}

All that remains now is to show that, as in Theorem~\ref{newbound}, 
to minimize $|A+A|$ we should make $m$ as large as possible, i.e., $m=t-1$.
The proof is by induction on $t-m$:

\begin{itemize}

\item
Suppose $m=t-1$. We check that (\ref{setting}) holds.
$$ \frac{2m - (s+1)}{2} \cdot\frac{|G|}{2^{m+1}} \;<\; 
|A|-\frac{|G|}{2^{m}} \;\leq\; \frac{2m - s}{2}\cdot\frac{|G|}{2^{m+1}} $$
Denote $k = \dim G - m - 1$ and $j = m - s$, and observe that $0 \leq j\leq m$.
Then the above expression for the minimal $|A+A|$ becomes
$$ |A+A| = \frac{m+1}{2^m}\cdot|G| + 
\left[\binom{m}{2} - \frac12 \binom{m - \left(m - s\right)}{2}\right] \cdot  \frac{|G|}{2^{m+1}}
= \frac{\binom{t}{2} + t + 1 - \frac12\binom{s}{2}}{2^t} \cdot |G| $$

\item 
Suppose $m<t-1$. 
The above discussion yields a compressed set $A$, 
such that $a_1...a_m$ is the quasi-fair quasi-dyadic partition of $|A|-|H|$,
and $|A+A|$ is minimal given $m$, and equals:
$$ |A+A| = \frac{m+1}{2^m}|G| + \sum_{1 \leq i<j \leq m}a_i \circ a_j $$
We show that increasing $m$ makes $|A+A|$ smaller.
Denote by $A'$, $H'$ and $a_1'...a_{m+1}'$ the corresponding set, subgroup and partition for $m'=m+1$.
Similarly:
$$ |A'+A'| = \frac{m+2}{2^{m+1}}|G| + \sum_{1 \leq i<j \leq m+1}a_i' \circ a_j' $$
Now define $a_0 = |H'| = |H|/2 = |G|/2^{m+1}$.
Since $a_1...a_m$ is a quasi-dyadic $m$-partition for $m>1$ and $a_1 < |H|$, 
necessarily $a_0 \geq a_i$ and $a_0 \circ a_i = a_0$ for all $1 \leq i \leq m$.
Hence for the quasi-dyadic $(m+1)$-partition $a_0...a_m$:
$$ |A+A| = \left(\frac{m+1}{2^m} - \frac{m}{2^{m+1}}\right)|G| + m \cdot a_0 + \sum_{1 \leq i<j \leq m}a_i \circ a_j = 
\frac{m+2}{2^{m+1}}|G| + \sum_{0 \leq i<j \leq m}a_i \circ a_j $$
But by Lemma~\ref{quasifair}, the partition $a_1'...a_{m+1}'$ gives the minimal value for this expression.
Moreover, since $a_0 = |H'| > a_1'$ these partitions differ and $|A'+A'| < |A+A|$.
\end{itemize}
\end{proof}

\begin{rmk*}
An examination of the proof reveals two kinds of reduction steps. 
Either $A$ is compressed without changing $\left\langle A \right\rangle$, 
or we find a set $A'$ where $|A'|=|A|$ and $|A'+A'|$ is substantially smaller than $|A+A|$.
Hence, the proof actually provides a characterization of the extremal case, 
up to compressions that preserve $\left\langle A \right\rangle$ and $|A+A|$.
\end{rmk*}

\section{Addition of two different sets}\label{diffsect}

What is the smallest possible cardinality of $A+B$ 
if $A,B \subseteq G = \Z_2^n$ are two affinely spanning subsets of given cardinalities? 
In this section we prove Theorem~\ref{AB}, which gives an essentially complete answer. 
In addition we establish a new isoperimetric inequality, which is used in the proof.
But first, we make some remarks concerning the theorem.

\begin{figure}[hbt]
\begin{center}
\includegraphics[width=1.00\textwidth,natwidth=2445,natheight=1845]{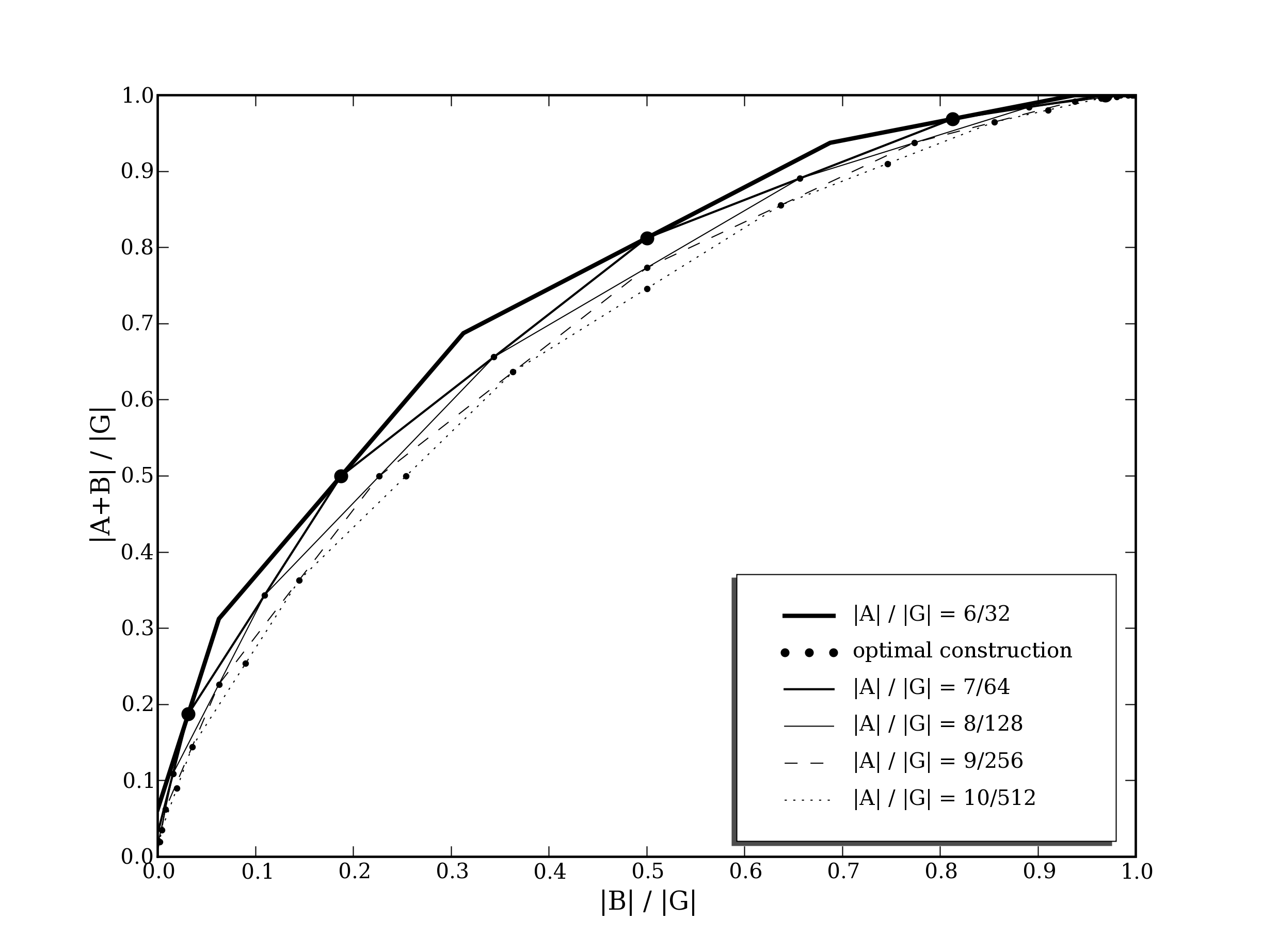}
\end{center}
\caption{An illustration of the lower bound}
\label{fig_ab}
\end{figure}

\begin{rmks*}[on Theorem~\ref{AB}]
$ $
\begin{enumerate}
\item 
Tightness: Consider $\left(|A|/|G|,|B|/|G|,|A+B|/|G|\right)$ as a point in $[0,1]^3$.
The Hamming balls construction shows that the bound goes through the points of the form:
$$ \left( \;\frac{1+t}{2^t}\;, \;\;\frac{1+t+...+\binom{t}{k}}{2^t}\;, \;\;\frac{1+t+...+\binom{t}{k+1}}{2^t} \;\right) \;\;\;\;\;\;\;\;\;\;\; 0 \leq k < t $$
An inspection of Figure~\ref{fig_ab} shows that 
all points properly inside their convex hull are strictly below the bound,
and hence cannot be realized by such sets.
In other words, further improvements of the bound will be local in nature.
\item 
The formulation of the theorem apparently breaks the symmetry 
and doesn't require $\left\langle B \right\rangle = G$.
Still, there is an asymmetry in the result as well, 
and the theorem is of interest mostly when $|A| \leq |B|$.
See also the remark after the proof.
\item
In order to simplify the statement of the theorem, 
$t$ is defined as the largest positive integer such that $|A|/|G| \leq (t+1)/2^t$.
However, the only assumption on $t$ which the proof actually uses is: 
$$\frac{t+2}{2^{t+1}} < \frac{|A|}{|G|}$$
The theorem can, therefore, be applied as well with $t$ larger than in the given formulation.
As Figure~\ref{fig_ab} shows, the resulting bound would be weaker, but
may still be useful in certain contexts.
\end{enumerate}
\end{rmks*}

Theorem~\ref{AB} implies that
a large enough number of large enough affinely generating sets must add up to the whole group:

\begin{cor}
\label{repeated}
Suppose that
$\left\langle A_1\right\rangle = \left\langle A_2\right\rangle = ... = \left\langle A_m\right\rangle = G = \Z_2^n$
with $|A_i|/|G| > (m+2)/2^{m+1}$ for all $i$. Then $A_1 + A_2 + ... + A_m = G$.
\end{cor}

\begin{proof}
We repeatedly apply the theorem with  $A = A_i$ and $B = A_1 + ... + A_{i-1}$ for all $1 \leq i \leq m$ to conclude
$$ \frac{|A_1 + A_2 + ... + A_i|}{|G|} 
> \frac{\binom{m+1}{0} + \binom{m+1}{1} + ... + \binom{m+1}{i}}{2^{m+1}} $$
Indeed, in view of remark 3 above and the assumption on the cardinalities, we may choose $t = m+1$,  
and then $k = i-1$ and $w > 0$ by the induction hypothesis.
Since $|A_1 + A_2 + ... + A_{m-1}| + |A_m| > |G|$
the proof is completed by the pigeonhole principle, $|A|+|B|>|G| \;\Rightarrow\;A+B=G$.
\end{proof}

The special case of Corollary~\ref{repeated} where all $A_i$ are identical
is due to Lev~\cite{lev_repeated}, following a conjecture of Zemor~\cite{zemor_repeated}.
Taking $A_i = D_1^{m+1} \times \Z_2^{n-m+1}$ for each $i$ shows that the assumption on the cardinalities is sharp.

\subsection{An Isoperimetric Inequality}

We are inspired by Frankl's short inductive proof~\cite{frankl} of Harper's theorem~\cite{harper}.

\begin{thm}[Harper's Inequality]\label{harperthm} Suppose $A \subseteq \Z_2^n$.\\
If for $1 \leq k \leq n$ integer and $0 \leq p \leq 1$ real
$$ |A| = \binom{n}{n} + \binom{n}{n-1} + ... + \binom{n}{k+1} + p \binom{n}{k} $$
then
$$ |A+D_1^n| \geq \binom{n}{n} + \binom{n}{n-1} + ... + \binom{n}{k} + p \binom{n}{k-1} $$
\end{thm}

In simple terms this theorem says that Hamming balls solve the vertex-isoperimetric problem in the hypercube.
However, it also deals, to varying degrees depending on the version of the
theorem, with sets of cardinalities strictly between $|D_{k-1}^n|$ and $|D_{k}^n|$.
A stronger version would replace the last summand of each expression with
$\binom{x}{k}$ and $\binom{x}{k-1}$ respectively, where $x \in [k,n]$ is real.
The optimal formulation due to Katona~\cite{katona} and Kruskal~\cite{kruskal} is stated in terms of
the $k$-cascade representations $\binom{a_k}{k} + \binom{a_{k-1}}{k-1} + ...$ 
and $\binom{a_k}{k-1} + \binom{a_{k-1}}{k-2} + ...$ respectively.
Frankl's method yields all three formulations.

Frankl's proof employs several useful operators on set-systems.
As usual, we freely move between the set-theoretic terminology of $2^{[n]}$ and the algebraic language of $\Z_2^n$.
The push-down operator $T_i$ and the shift operator $S_{ij}$ have already appeared in Section~\ref{toolsect}.
The upper and the lower \emph{shadow} operators act on a set-system ${\cal F} \subseteq 2^{[n]}$ by
\begin{align*}
\delta \mathcal{F} &= \left\{ J \cup \{i\} \;\left|\; J \in \mathcal{F}, i \notin J \right. \right\} \\
\partial \mathcal{F} &= \left\{ J \setminus \{i\} \;\left|\; J \in \mathcal{F}, i \in J \right. \right\}
\end{align*}
respectively.
For downsets, the notion of the shadow is close to that of the neighborhood in the theorem.
If $C \subseteq \Z_2^n$ is a non-empty downset, then $C + D_1^n = \delta C \cup \{0\}$.
Note that always $0 \notin \delta A$.
Another useful operation on set-systems is \emph{classification} by~$n$, denoted by:
\begin{align*}
\mathcal{F}^- &= \left\{ J \;\left|\; J \in \mathcal{F}, n \notin J \right. \right\} \\
\mathcal{F}^+ &= \left\{ J \setminus \{n\} \;\left|\; J \in \mathcal{F}, n \in J \right. \right\}
\end{align*}
When $A \subseteq \Z_2^n$, we regard $A^+$ and $A^-$ as subsets of $\Z_2^{n-1}$.

Following Frankl~\cite{frankl}, we proceed with two lemmas regarding properties of shifts and shadows.

\begin{lemma}\label{frankl_lemma}
Suppose $C \subseteq \Z_2^n$ is a shift-minimal downset. 
\begin{enumerate}
\item[(1)] $\delta (C^+) \subseteq (\delta C)^+ = C^-$ with equality iff $C = \varnothing$
\item[(2)] $\delta (C^-) = (\delta C)^-$
\end{enumerate}
\end{lemma}

\begin{proof}
Examine the effect of the operators on the representation of some $x \in C$ 
with the standard basis $e_1,...,e_n$.

In both $\delta (C^+)$ and $(\delta C)^+$, some $e_i$ is added and $e_n$ is removed.
However, in $\delta (C^+)$ certainly $i \neq n$ since $C^+$ lives in $\Z_2^{n-1}$,
while in $(\delta C)^+$ it is possible that $i=n$.
Hence $\delta (C^+) \subseteq (\delta C)^+$.
By shift-minimality $C$ is closed under these swaps, thus $(\delta C)^+ \subseteq C^-$.
Moreover, every element of $C^-$ is obtained by adding $e_n$ and then deleting it, so there is equality.
However, $\delta (C^+)$ is strictly smaller since $0 \in C^- \setminus \delta (C^+)$ unless $C$ is empty.

For $\delta (C^-) = (\delta C)^-$, 
note that both sets consist of elements of the form $x + e_i$ for $x \in C$ and $i < n$,
where $e_i$ and $e_n$ do not appear in $x$'s standard representation.
\end{proof}

The following lemma is well known. See e.g.~\cite{frankl_short,katona_intersection}.
Here we prove it as a special case of the compression machinery. 

\begin{lemma}\label{shiftshadows}
For all $A \subseteq \Z_2^n$ and $1 \leq i, j \leq n$ such that $i \neq j$,
\begin{enumerate}
\item[(1)] $\delta \left(S_{ij} A \right) \subseteq S_{ij}\left( \delta A \right)$
\item[(2)] $\partial \left(S_{ij} A \right) \subseteq S_{ij}\left( \partial A \right)$
\end{enumerate}
\end{lemma}

\begin{proof}
By passing from $A$ to $\sum_i e_i - A$, it is enough to prove only one of the inclusions.
Denote $A = \bigcup_{k=0}^{n} A_k$ where $A_k = A \cap \left(D_k^n \setminus D_{k-1}^n\right)$.
Note that we can work with each $A_k$ separately.
One can write 
$$ \delta \left( S_{ij}A_k \right) = \left(D_1^n + C_{ij} \left( A_k \cup D_{k-1}^n \right)\right) \setminus D_k^n ,$$
and 
$$ S_{ij}\left(\delta A_k\right) = C_{ij} \left( D_1^n + \left( A_k \cup D_{k-1}^n \right) \right) \setminus D_k^n ,$$
yielding our claim by Lemma~\ref{sumcomp}, 
since $D_1^n + C_{ij}(B) = C_{ij}(D_1^n) + C_{ij}(B) \subseteq C_{ij}(D_1^n + B)$.
\end{proof}

Our isoperimetric inequality concerns a family of non-empty downsets $C_1 ... C_l \subseteq \Z_2^n$, 
rather than a single one.
For the volume and the shadow we take the average quantities, denoted by: 
$$ \mathrm{E}\left[C\right] = \frac{1}{l}\sum\limits_{m=1}^{l}\left|C_m\right| \;\;\;\;\;\;\;\; 
\mathrm{E}\left[\delta C\right] = \frac{1}{l}\sum\limits_{m=1}^{l}\left|\delta(C_m)\right| $$
It is hard to make a meaningful statement about these average quantities without limiting the downsets somehow. 
To see this, consider what happens when each $C_m$ is either full or empty.
We limit the variability of the downsets by assuming the {\em antichain condition}. 
Namely, we require that for each $i$ and $j$, 
$C_i \setminus C_j$ is an antichain with respect to set-systems inclusion, or equivalently $C_j \supseteq \partial C_i$. 

\begin{prop}\label{iso1}
Suppose $C_1 ... C_l \subseteq \Z_2^n$ is a family of downsets which satisfies the antichain condition.
If
$$ \mathrm{E}[C] = \binom{n}{0} + \binom{n}{1} + ... + \binom{n}{k-1} + p \binom{n}{k} $$
for some integer $k \geq 0$ and real number $0 \leq p < 1$, then
$$ \mathrm{E}\left[\delta C\right] \geq \binom{n}{1} + \binom{n}{2} + ... + \binom{n}{k} + p \binom{n}{k+1} $$
\end{prop}

Since for non-empty downsets $C + D_1^n = \{0\} \cup \delta C$,
the corresponding inequality in the language of neighborhoods is as follows.

\begin{cor}\label{iso2}
In the setting of Proposition~\ref{iso1}, if $C_1 ... C_l$ are non-empty then
$$ \mathrm{E}\left[C + D_1^n\right] \geq \binom{n}{0} + \binom{n}{1} + ... + \binom{n}{k} + p \binom{n}{k+1} $$
\end{cor}

\begin{proof}(of Proposition~\ref{iso1})
We may assume that the downsets are shift-minimal.
Indeed, for each downset $C_m$ clearly $S_{ij}C_m$ is a downset of the same size, while
$\left|\delta \left(S_{ij}C_m\right)\right| \leq \left|S_{ij}\left(\delta C_m\right)\right| = \left|\delta C_m\right|$ 
by Lemma~\ref{shiftshadows}.
If $C_{m'} \setminus C_m$ is an antichain, then $C_m \supseteq \partial C_{m'}$, 
hence $S_{ij}C_m \supseteq S_{ij}\left(\partial C_{m'}\right) \supseteq \partial\left(S_{ij}C_{m'}\right)$
by Lemma~\ref{shiftshadows} again, and hence $S_{ij}C_{m'} \setminus S_{ij}C_m$ is an antichain as well.
In conclusion, $S_{ij}C_1 ... S_{ij}C_l$ satisfy the antichain condition,
$\mathrm{E}[C] = \mathrm{E}[S_{ij}C]$ and $\mathrm{E}[\delta C] \geq \mathrm{E}[\delta (S_{ij}C)]$.
After a finite sequence of shifts the downsets are all shift-minimal,
since for a proper shift $\sum_{m}\hbar\left(S_{ij}C_m\right) < \sum_{m}\hbar\left(C_m\right)$.

The case $k=0$ is established separately.
Note that in this case $\mathrm{E}[C] < 1$, hence $C_m = \varnothing$ for some $m$.
Actually, this is a sufficient condition for $k=0$, 
because all other downsets are either $\varnothing$ or $\{0\}$ by the antichain condition.
Since $\delta \{0\} = \left\{e_1,...,e_n\right\}$, 
clearly $\mathrm{E}\left[\delta C\right] = n \cdot \mathrm{E}[C]$ as required.

Following Frankl, we proceed by induction on $n$.
By convention $\binom{n}{k} = 0$ for $n < k$.
Thus, for $n=0$ the lemma is vacuously satisfied by $\mathrm{E}\left[\delta C\right] \geq 0$.

For positive $k$ and $n$, we employ the induction hypothesis on the families $C_1^- ... C_l^-$ 
and $C_1^+ ... C_l^+$ in $\Z_2^{n-1}$.
It is easily checked that given a downset $C_m$, the
sets $C_m^+$ and $C_m^-$ are downsets as well.
In addition, if $C_{m'} \setminus C_m$ is an antichain, 
then so are its two parts, $C_{m'}^- \setminus C_m^-$ and $C_{m'}^+ \setminus C_m^+$,
hence the new families satisfy the antichain condition.

By the induction hypothesis on $C_1^+ ... C_l^+ \subseteq \Z_2^{n-1}$, at least one of the following must hold:
\begin{align*}
&\mathrm{E}\left[C^+\right] < 
\binom{n-1}{0} + \binom{n-1}{1} + ... + \binom{n-1}{k-2} + p \binom{n-1}{k-1} \\
&\mathrm{E}\left[\delta\left(C^+\right)\right] \geq 
\binom{n-1}{1} + \binom{n-1}{2} +...+ \binom{n-1}{k-1} + p\binom{n-1}{k}
\end{align*}
Use $\mathrm{E}\left[C^-\right] = \mathrm{E}[C] - \mathrm{E}\left[C^+\right]$ and Pascal's rule in the first case, 
or $\mathrm{E}\left[C^-\right] \geq 1 + \mathrm{E}\left[\delta\left(C^+\right)\right]$ 
by Lemma~\ref{frankl_lemma}(1) in the second one, to deduce:
$$ \mathrm{E}\left[C^-\right] \geq \binom{n-1}{0} + \binom{n-1}{1} + ... + \binom{n-1}{k-1} + p \binom{n-1}{k} $$
Note that since $k > 0$ each $C_m$ is non-empty, so there is proper inclusion in the lemma, 
which yields the extra $1$ in the calculation.
By the induction hypothesis on $C_1^- ... C_l^- \subseteq \Z_2^{n-1}$:
$$ \mathrm{E}\left[\delta\left(C^-\right)\right] \geq 
\binom{n-1}{1} + \binom{n-1}{2} + ... + \binom{n-1}{k} + p \binom{n-1}{k+1} $$
By Lemma~\ref{frankl_lemma}, 
$ \mathrm{E}\big{[}\delta C\big{]} = 
\mathrm{E}\big{[}\left(\delta C\right)^-\big{]} + \mathrm{E}\big{[}\left(\delta C\right)^+\big{]} = 
\mathrm{E}\big{[}\delta\left( C^- \right)\big{]} + \mathrm{E}\big{[} C^- \big{]} $,
hence by Pascal's rule:
$$ \mathrm{E}\left[\delta C\right] \geq \binom{n}{1} + \binom{n}{2} + ... + \binom{n}{k} + p \binom{n}{k+1} $$
\end{proof}

\subsection{Proof of Lower Bound}

\begin{proof}(of Theorem~\ref{AB})
The general idea is similar to the case $A+A$ discussed in the previous section.
By applying various compressions, the sets $A$ and $B$ acquire certain structural properties.
These, in turn, allow us to derive estimates on the cardinality of $A+B$.

Lemma~\ref{sumcomp} asserts that compressions do not increase sumsets:
$\left|C_I(A)+C_I(B)\right| \leq \left|A+B\right|$ holds while $|C_I(A)|=|A|$ and $|C_I(B)|=|B|$.
Thus, in the search for a lower bound for $|A+B|$,
one can first apply a compression $C_I$ on $A$ and $B$ simultaneously. 
Since $\left\langle A \right\rangle = G$, we may suppose $E = \{0,e_1,e_2,...,e_n\} \subseteq A$
and restrict ourselves only to compressions that preserve the inclusion $E \subseteq A$.
By Lemma~\ref{comp}(3), if a compression $C_I$ changes either $A$ or $B$, 
then $\hbar(A)+\hbar(B)$ strictly decreases. 
It follows that every sequence of such compressions must terminate.
In conclusion, we can assume that both $A$ and $B$ are invariant under these compressions,
or for short $\langle\langle E \subseteq A \rangle\rangle$-\emph{compressed}.
This implies that $B$ is $I$-compressed for every $I \subseteq [n]$ such that $A$ is $I$-compressed.

Lemma~\ref{structure} provides a description of $A$ under this assumption. 
In particular,
$H \subseteq A \subseteq H+E$ for some subgroup $H = \ip{0,e_1,...,e_h}$.
We next derive some structural properties of $B$.

\begin{lemma}\label{structureB}
Suppose $A, B \subseteq G = \Z_2^n$ are $\langle\langle E \subseteq A \rangle\rangle$-compressed.
Let $H \subseteq A$ be as in Lemma~\ref{structure}.
Consider $G/H \cong \Z_2^m$ where $m = n - h = \mathrm{codim}\;H$,
with the basis $\left\{e_{h+1} + H, ..., e_{h+m} + H\right\}$ and the partial order of the corresponding set-system. 
For $1 \leq j \leq |H|$ let
$$ C_j = \left\{ H' \in G/H \;\Big{|}\; |B \cap H'| \geq j \right\} $$
Then $C_1...C_{|H|}$ are downsets, and satisfy the antichain condition.
\end{lemma}

\begin{proof}
By Lemma~\ref{structure}(3), $A$ is $\{1,...,h,h+i\}$-compressed for $1 \leq i \leq m$,
and therefore so is $B$.

Let $H' \prec H''$ be adjacent $H$-cosets in the partial order. 
$H'' = e_{h+i} + H'$ for some $1 \leq i \leq m$.
Since $B$ is $\{1,...,h,h+i\}$-compressed, $B \cap (H' \cup H'')$ must be an initial segment of $H' \cup H''$.
Note that all $H'$ elements are lexicographically smaller than those of $H''$.
Consequently, if $B \cap H'' \neq \varnothing$ then necessarily $H' \subseteq B$.
In other words, $H'' \in C_1 \;\Rightarrow\; H' \in C_{|H|}$ for each such pair.

In particular, $C_j$ is a downset because 
$H'' \in C_j \subseteq C_1 \Rightarrow H' \in C_{|H|} \subseteq C_j$,
and $C_j \setminus C_k$ is an antichain since 
$C_j \setminus C_k  \subseteq C_1 \setminus C_{|H|} \not\supseteq \{H',H''\}$.
\end{proof}

We can conclude now the proof of Theorem~\ref{AB} in the following three steps:
\begin{enumerate}
\item 
We use the structure of the compressed sets to find new expressions for the cardinalities of $B$ and $A+B$.
Let $C_1 ... C_{|H|}$ be as in Lemma~\ref{structureB}.
By interchanging the order of summation:
$$ |B| = \sum\limits_{H' \in G/H} |B \cap H'| = 
\sum\limits_{H' \in G/H}\#\left\{j \in \N \;\Big{|}\; |B \cap H'| \geq j\right\} \;=\; \sum\limits_{j=1}^{|H|}|C_j| $$
We estimate $|A+B|$ in a similar fashion.
For $1 \leq j \leq |H|$, suppose $H'' \in \delta(C_j) \cup \{H\}$.
We show that $A+B$ intersects $H''$ in $j$ elements at the least:
\begin{itemize}
\item 
If $H''=H$, use $H \subseteq A$ and $0 \in B \neq \varnothing$ to obtain
$ \left|(A+B) \cap H''\right| \geq \left|(H+0) \cap H\right| \geq j$.
\item
Otherwise $H'' = e_{h+i} + H'$ for some $H' \in C_j$ and $1 \leq i \leq m  = \mathrm{codim}\;H$.
Since $e_{h+i} \in E \subseteq A$, clearly
$ \left|(A+B) \cap H''\right| \geq \left|(e_{h+i}+B) \cap (e_{h+i} + H')\right| = |B \cap H'| \geq j $.
\end{itemize}
Consequently:
$$ |A+B| = 
\sum\limits_{H'' \in G/H}\#\left\{j \in \N \;\Big{|}\; |(A+B) \cap H''| \geq j\right\} \;\geq\;
\sum\limits_{j=1}^{|H|}|\delta(C_j) \cup \{H\}| $$

\item
We use the isoperimetric inequality in order to obtain a lower bound on $|A+B|$ given $m$ and $|B|$.
Let $0 \leq k \leq m$ and $w \in [-1,1]$ be such that:
$$ |B| = \frac{ \binom{m+1}{0} + \binom{m+1}{1} + ... + \binom{m+1}{k} + w \binom{m}{k} }{2^{m+1}} \cdot |G|$$
We substitute $|B| = \sum |C_j|$ in the left-hand side, 
apply Pascal's rule to $\binom{m+1}{1}...\binom{m+1}{k}$ on the right-hand side, 
and divide both by $|H| = |G|/2^m$, to obtain:
$$ \mathrm{E}[C] = \frac{1}{|H|} \sum\limits_{j=1}^{|H|}|C_j| = \frac{|B|}{|H|} = 
\binom{m}{0} + \binom{m}{1} + ... + \binom{m}{k-1} + \frac{1+w}{2} \binom{m}{k} $$
Now by Proposition~\ref{iso1}
$$ \mathrm{E}\left[\{H\} \cup \delta C\right] \;\geq\; 
1 + \binom{m}{1} + \binom{m}{2} + ... + \binom{m}{k} + \frac{1+w}{2} \binom{m}{k+1} $$
where the union is disjoint since always $H \notin \delta C_j$.
In terms of $A$ and $B$, this implies:
$$ |A+B| \geq \frac{\binom{m+1}{0} + \binom{m+1}{1} + ... + \binom{m+1}{k} + \binom{m+1}{k+1} + w \binom{m}{k+1}}{2^{m+1}} \cdot |G|$$

\item
What values can $m = \mathrm{codim}\;H$ take?
By Lemma~\ref{structure}(7),
$\left(1 + \frac{m}{2}\right) / 2^m \geq |A|/|G|$,
where the case $m=1$ is separately deduced from the assumption $|A|/|G| \leq 3/4$.
On the other hand, by the theorem's assumption on $t$, 
$|A|/|G| > (t+2)/2^{t+1} = \left(1 + \frac{t}{2}\right) / 2^t$.
Since the sequence $\left(1 + \frac{n}{2}\right) / 2^n$ is monotone, we infer $m < t$.

The theorem is obtained by plugging $m = t-1$ into the derived lower bound.
We claim that for smaller $m$ the bound is even higher, as demonstrated in Figure~\ref{fig_ab}.
Indeed, for each $1 \leq i \leq t-2$, the graph of the lower bound on $|A+B|$ given $m=i$ is
concave down by the log-concavity of the binomial coefficients, 
$\binom{m}{k}/\binom{m}{k-1} \geq \binom{m}{k+1}/\binom{m}{k}$.
Thus the graph of $m=i+1$, which connects the midpoints of adjacent segments in the $m=i$ graph, must be lower.
\end{enumerate}
\end{proof}

\begin{rmk*}
By the Hamming balls construction,
the lower bound we have found is optimal on a biparametric discrete family of points.
In view of our treatment of $F(K)$ in the previous section, we expect
that at intermediate points better bounds should be provable.

There are three points where our approach to Theorem~\ref{AB} may be suboptimal: 
the isoperimetric inequality we use is not always perfectly tight, 
the addition of $H \cup E$ instead of the whole of $A$,
and dropping the assumption on $B$'s affine span.

It is perhaps worth remarking that the machinery of compressions can still be applied
under the assumption $\left\langle A \right\rangle = \left\langle B \right\rangle = G$.
This is done by showing that without loss of generality we may assume that $A$ and $B$ 
are simultaneously compressed such that they include a common affine basis.

Here is a brief outline of how this is done.
First, partition $A$ and $B$ into their intersections with cosets of $\langle (A-A)\cap(B-B) \rangle$.
These parts can be translated without increasing $|A+B|$, such that $A-A$ and $B-B$ include a common basis of $G$.
Then apply $\{i\}$-compressions with respect to this basis, until it is included in $A \cap B$.
\end{rmk*}

\section{Acknowledgment}\label{ack}

I would like to thank my advisor, Professor Nati Linial, 
for his patient and helpful guidance during the research and the preparation of this manuscript.

{
\scriptsize
\bibliographystyle{plain}
\bibliography{sumsofsets}
}

\end{document}